\documentclass[a4paper,12pt]{article}
\usepackage{dsfont}
\usepackage[T1]{fontenc}
\usepackage{amsmath,amssymb,amsthm}
\usepackage[dvips]{graphicx}
\usepackage[english]{babel}
\usepackage{color}

\theoremstyle{plain}
\newtheorem{theorem}{Theorem}[section]

\newtheorem{proposition}{Proposition}[section]
\newtheorem{coro}{Corollary}[section]
\newtheorem{lemma}{Lemma}[section]

\theoremstyle{definition}
\newtheorem{defi}{Definition}[section]
\newtheorem{ex}{Example}[section]

\theoremstyle{remark}
\newtheorem{remark}{Remark}[section]

\def\noeop{\let\popQED\relax}
\newcommand{\reel}{\mathbb{R}}

\newcommand{\Z}{\mathbb{Z}}

\newcommand{\n}{\mathbb{N}}

\newcommand{\e}{\mathbb{E}}

\newcommand{\indic}{\mathds{1}}
\DeclareMathOperator{\Cov}{Cov}

\begin{document}
\baselineskip 6mm

\title{Products of random variables and the first digit phenomenon}

\author{Nicolas Chenavier\,\textsuperscript{1,2}, Bruno Mass\'e\,\textsuperscript{1,2,3}  and Dominique Schneider\,\textsuperscript{1,2}}

\footnotetext[1]{Univ. Littoral C\^ote d'Opale, EA 2597 - Laboratoire de mathématiques pures et appliqu\'ees Joseph Liouville, F-62228 Calais, France.}
\footnotetext[2]{e-mail: \{nicolas.chenavier, bruno.masse, dominique.schneider\}@lmpa.univ-littoral.fr.}
\footnotetext[3]{Corresponding author.}

\date{\today}
\maketitle

\begin{abstract}
We provide conditions on dependent and on non-stationary random variables $X_n$ ensuring that the mantissa of the sequence of products $\left(\prod_{1}^{n}X_k\right)$ is almost surely distributed following the Benford's law or converges in distribution to the Benford's law. 
\end{abstract}

{\bf Keywords:} Benford's law, density, mantissa, weak convergence

{\bf AMS classification:} 60B10,11B05,11K99

\section{Introduction}

Let $b>1$. The {\it Benford's law in base $b$} is the probability measure $\mu_b$ on the interval $[1;b[$ defined by 
$$
\mu_b([1;a[\,)= \log_{b}a\;\;\;\;\;(1\leq a <b)
$$
where $\log_{b}a$ denotes the logarithm in base $b$ of $a$.
The {\it mantissa} in base $b$ of a positive real number $x$ is the unique number ${\cal M}_b(x)$ in $[1;b[\,$ such that there exists an integer $k$ satisfying $x={\cal M}_b(x)b^k$. 

When a sequence of positive random variables $(X_n)$ is of a type usually considered by the probabilists and the statisticians, there is little to be said on $({\cal M}_b(X_n))$ (see Remark \ref{rem} in Section \ref{ud1} for instance) while by contrast there is much to report on $\left({\cal M}_b\left(\prod_{1}^{n}X_k\right)\right)$ as we will see. Our purpose is therefore to exhibit conditions on $X_n$ ensuring that the sequence $\left({\cal M}_b\left(\prod_{1}^{n}X_k\right)\right)$ is almost surely distributed following $\mu_b$ (see Definition \ref{def1}) or ensuring that the law of ${\cal M}_b\left(\prod_{1}^{n}X_k\right)$ converges weakly to $\mu_b$ as $n\rightarrow +\infty$. We hope that this will enlarge, to a certain extent, the field of applications of the Benford's law (see Section \ref{fdp} for examples of such applications).

To the best of our knowledge, apart from \cite{miller}, the known results on the asymptotic behaviour of $\left({\cal M}_b\left(\prod_{1}^{n}X_k\right)\right)$ only deal with the cases where the $X_n$ are independent and identically distributed and the situations where $X_n=X$ for $n\geq 1$ and $X$ is some random variable (see Section \ref{known} for details).

\subsection{The first digit phenomenon}\label{fdp}

Benford \cite{benford} noticed in 1938 that many real-life lists of numbers have a strange property: numbers whose mantissae are small are more numerous than those whose mantissae are large. This fact is called the {\it First Digit Phenomenon}. He also noticed that this phenomenon seems independent of the units. This led him to make a scale-invariance hypothesis (more or less satisfied in real life) from which he derived that $\mu_{10}$ can be seen as the (ideal) distribution of digits or mantissa of many real-life numbers. Of course, this ideal distribution is never achieved in practice. 

Several mathematicians have been involved in this subject and have provided sequences of positive numbers whose mantissae are (or approach to be) distributed following $\mu_b$ in the sense of the natural density \cite{anderson,cohen,diaconis,eliahou,masse4} (see Definition \ref{def1}), random variables whose mantissa law is or approaches $\mu_b$ \cite{berger,dumbgen,hill,janvresse,leemis}, sequences of random variables whose mantissae laws converge to $\mu_b$ or whose mantissae are almost surely distributed following $\mu_b$ \cite{masse2,miller,robbins,sharpe}. Among the many applications of the First Digit Phenomenon, we can quote: fraud detection \cite{nigrini}, computer design \cite{hamming,knuth} (data storage and roundoff errors), image processing \cite{xu} and data analysis in natural sciences \cite{nigrini0,sambridge}. 

\subsection{Content}

Section \ref{preli} is devoted to notation, definitions and tools from Uniform Distribution Theory. Our main results are presented in Section \ref{gene}. Theorem \ref{th1} gives a simple necessary and sufficient condition ensuring that the sequence $\left({\cal M}_b\left(\prod_{1}^{n}X_k\right)\right)$ is almost surely distributed following $\mu_b$ under the assumption that the sequence $(X_n)$ is stationary. Theorem \ref{th2}  gives a sufficient condition without constraints on the dependence and on the stationarity of the $X_n$. These properties are used in Section \ref{appli} to investigate the cases where the random variables $X_n$ are stationary and log-normal, are exchangeable, are stationary and 1-dependent and the case where they  are independent and non-stationary. We provide in Section \ref{app} a survey of the main known properties of the Benford's law (scale-invariance, power-invariance and invariance under mixtures). We think that this might help, together with Section \ref{known}, to put our results in perspective.

\subsection{Known results on random product sequences}\label{known}

When the $X_n$ are independent and identically distributed, 
\begin{itemize}

\item the sequence $\left({\cal M}_b\left(\prod_{1}^{n}X_k\right)\right)$ is almost surely distributed following $\mu_b$ if and only if for every positive integer $h$
\[
\e(\exp(2i\pi h\log_b X_1))\neq 1,
\]
that is to say if and only if the common law of the $X_n$ is not supported by any set $\{b^{\frac{z}{h}} : z \;\mbox{integer}\}$ ($h$ positive integer);

\item the law of ${\cal M}_b\left(\prod_{1}^{n}X_k\right)$ converges weakly to $\mu_b$ if and only if for every positive integer $h$
\[
|\,\e(\exp(2i\pi h\log_b X_1))|\neq 1,
\]
that is to say if and only if the common law of the $X_n$ is not supported by any set $\{b^{a+\frac{z}{h}} : z \;\mbox{integer}\}$ ($a\in[0,1[$, $h$ positive integer).
\end{itemize}
A proof of the first statement is available in \cite{robbins}. The second statement is a direct consequence of Lemma \ref{lem1} below. 

Moreover 

\begin{itemize}

\item the sequence $\left({\cal M}_b(X^n)\right)$ is almost surely distributed following $\mu_b$ if and only $P(X\in\{b^r : r \;\mbox{rational}\})=0$;

\item the law of ${\cal M}_b(X^n)$ converges weakly to $\mu_b$ if and only if for every positive integer $h$
\[
\lim_{n\rightarrow \infty}\e(\exp(2i\pi nh\log_b X))=0.
\]
\end{itemize}
The two above conditions are fulfilled when $X$ (and hence $\log_b X$) admits a density.
The first statement derives from the fact that the sequence $({\cal M}_b(c^n))$ is distributed following $\mu_b$ if and only if $\log_b c$ is irrational (see \cite[p. 8]{kuipers} and Section \ref{ud1} below). The second statement is also a direct consequence of Lemma \ref{lem1} below. 

It is worth noting that, in the situations discussed above, the law of ${\cal M}_b\left(\prod_{1}^{n}X_k\right)$ converges weakly to $\mu_b$ in most cases for every value of $b$, while there does not exist any random variable $Z$ such that the law of ${\cal M}_b(Z)$ is $\mu_b$ for every value of $b$ (see Section \ref{base-inv}). 

The following example shows the kind of difficulties that can arise when the $X_n$ are neither independent nor stationary.

\begin{ex}
Consider an i.i.d. sequence of random variables $(Z_m)$ with common law $\mu_{b}$. Set $X_{2m-1}=Z_m$ and $X_{2m}=b/Z_m$ ($m\geq 1$). Then the random variables $X_n$ are identically distributed following $\mu_{b}$ because 
\[
P\left(\frac{b}{Z_m}<t\right)=P\left(\frac{b}{t}<Z_m\right)=1-\log_{b}\left(\frac{b}{t}\right)=\log_{b} t\;\;\;(1\leq t <b) .
\]

The sequence $({\cal M}_b(Z_m))$ is a.s. distributed following $\mu_{b}$ by Glivenko-Cantelli Theorem and ${\cal M}_b(b^{m-1}Z_m)={\cal M}_b(Z_m)$. Then the sequence $\left({\cal M}_b\left(\prod_{1}^{n}X_k\right)\right)$ is a.s. distributed following $(1/2)\delta_1+(1/2)\mu_{b}$ because $\prod_{1}^{2m}X_k=b^m$ and $\prod_{1}^{2m-1}X_k=b^{m-1}Z_m$. And
the law of ${\cal M}_b(\prod_{1}^{n}X_k)$ does not converge weakly as $n\rightarrow +\infty$ because 
the law of ${\cal M}_b(\prod_{1}^{2m}X_k)$ is $\delta_1$ while the law of ${\cal M}_b(\prod_{1}^{2m-1}X_k)$ is $\mu_b$.
\end{ex}

\section{Preliminaries}\label{preli}

We present here some useful notation and definitions and the relationship between the investigation of Benford's law and Uniform Distribution Theory.

\subsection{Other notation and definitions}

We shall consistently use the following notation through this paper: whenever $(X_n)$ is a given sequence of positive random variables, we set
\[
Y_n=\prod_{k=1}^{k=n}X_k\;\;\;\mbox{ and }\;\;\;Y_{m,n}=\prod_{k=m}^{k=n}X_k\;\;\;\;(1\leq m\leq n)\,.
\]

Her is some other notation used in this article: the natural logarithm is denoted by $\log$; for any real $x$ and any integer $h$, we set $e_h(x)=\exp(2i\pi hx)$ where $i^2=-1$; the symbol $\{x\}$ stands for the fractional part of a real $x$; we write $\Z^+$ for the set of positive integers; the standard abbreviations {\it a.s.}, {\it r.v.} and {\it i.i.d.} stand respectively for {\it almost surely} (or {\it almost sure}), {\it random variable} and {\it independent and identically distributed}; all the r.v.'s in consideration are supposed to be defined on the same probability space $(\Omega, {\cal T},P)$ and the law of a r.v. $Z$ is denoted $P_Z$. 

\begin{defi}\label{def1}
A sequence $(v_n)$ of real numbers in $[1;b[$ is called {\it Benford in base $b$} if it is distributed as $\mu_b$, that is to say if 
$$
\lim_{N\rightarrow +\infty}\frac{1}{N}\sum_{n=1}^{N}\mathds{1}_{[1;\,a[}(v_n)= \log_{b}a\;\;\;\;\;(1\leq a <b)\,.
$$ 
A sequence $(u_n)$ of positive numbers is also called Benford in base $b$ if the sequence of mantissae $\left({\cal M}_b(u_n)\right)$ is Benford in base $b$. 
\end{defi}

For instance, the sequences $(n!)$, $(n^n)$ and $(c^n)$ (with $\log_b c$ irrational) are Benford in base $b$ \cite{masse4}. The sequences $(n)$ and $(\log n)$ and the sequence of prime numbers are not \cite{fuchs}. See \cite{anderson,masse4} for more examples of Benford sequences.

\begin{defi}\label{def2}

We say that a positive random variable $Z$ is Benford in base $b$ when $P_{{\cal M}_b(Z)}=\mu_b$, that a sequence of positive random variables $(Z_n)$ is a.s. Benford in base $b$ when
$$
P(\{\omega : (Z_n(\omega)) \textrm{ is a Benford sequence in base $b$}\})=1
$$
and that $Z_n$ tends to be Benford in base $b$ when the sequence $(P_{{\cal M}_b(Z_n)})$ converges weakly to $\mu_b$.
\end{defi}

These notions are connected (see Section \ref{known}, Remark \ref{rem} and Theorem \ref{th1} and its corollaries) but are however significantly different one from the other. Indeed, suppose that $Z_n=n!$ a.s. for $n\geq 1$. Then the sequence $\left(P_{{\cal M}_{b}(Z_n)}\right)$ does not converge weakly while the sequence $(Z_n)$ is a.s. Benford in base $b$. Conversely, suppose that $Z_n=T$ ($n\geq 1$) where $T$ is Benford in base $b$. Then $Z_n$ tends to be Benford in base $b$, but the sequence $(Z_n)$ is a.s. not Benford.

\subsection{Benford law and uniform distribution modulo 1}\label{ud1}
It is well known and easy to verify that a sequence $(u_n)$ of positive numbers is Benford in base $b$ if and only if the sequence of fractional parts $(\{\log_b u_n\})$ is uniformly distributed in $[0,1[$, that is to say 
$$
\lim_{N\rightarrow +\infty}\frac{1}{N}\sum_{n=1}^{N}\mathds{1}_{[0,\,c[}(\{\log_b u_n\})= c\;\;\;\;\;(0\leq c <1)\,,
$$
that a positive random variable $Z$ is Benford in base $b$ if and only if $P_{\{\log_b Z\}}$ is the uniform probability on $[0,1[$ and that a sequence $(Z_n)$ of positive random variables tends to be Benford in base $b$ if and only if the sequence $(P_{\{\log_bZ_n\}})$ converges weakly to the uniform distribution in $[0,1[$. Combining this with the celebrated Weyl's Criterion \cite[p. 7]{kuipers} yields the following lemma.

\begin{lemma} \label{weyl}
Consider a sequence $(u_n)$ of positive numbers and a sequence $(Z_n)$ of positive r.v.'s. Then $(u_n)$ 
is Benford in base $b$ if and only if
\[
\forall\, h\in\Z^+\,,\;\lim_{N\rightarrow+\infty}\frac{1}{N}\sum_{n=1}^{N}e_h(\log_b u_n)=0
\]
and $(Z_n)$ is a.s. Benford in base $b$ if and only if
\[
\forall\, h\in\Z^+\,,\;\lim_{N\rightarrow+\infty}\frac{1}{N}\sum_{n=1}^{N}e_h(\log_b Z_n)=0\;\;\;\;\mbox{a.s.}.
\]
\end{lemma}

L\'evy's Theorem states that the weak convergence of a sequence $(\mu_n)$ of probability measures to a probability measure $\mu$ is equivalent to the pointwise convergence of the characteristic function of $\mu_n$ to that of $\mu$. On the torus $\reel/\Z$, the convergence of the Fourier coefficients suffices \cite[p. 363]{bill}. Since for every $x>0$ and every $h\in\Z^+$, $e_h(\{\log_b x\})=e_h(\log_b x)$, we get the following characterizations. 

\begin{lemma}\label{lem1} Consider a positive r.v. $Z$ and a sequence $(Z_n)$ of positive r.v.'s. Then $Z$ is Benford in base $b$ if and only if
$$
\forall\, h\in\Z^+\,,\;\e\left(e_h(\log_b Z)\right)=0
$$ 
and $Z_n$ tends to be Benford in base $b$ if and only if 
$$
\forall\, h\in\Z^+\,,\;\lim_{n\rightarrow+\infty}\e\left[e_h(\log_b Z_n)\right]=0.
$$
\end{lemma}

We are now able to treat the remark evoked at the beginning of Section \ref{known}.
\begin{remark}\label{rem}
Let $Z_1, Z_2,\dots$ be identically distributed positive random variables (independent or not, stationary or not). In particular, $Z_n$ tends to be Benford in base $b$ if and only if each $Z_n$ is Benford in base $b$. More interesting is the fact that, if $(Z_n)$ is a.s. Benford in base $b$, then each $Z_n$ must be Benford in base $b$. Indeed, for every $h\in\Z^+$ and $N\geq 1$, $\e(e_h(\log_b Z_1))=\e\left((1/N)\sum_1^Ne_h(\log_b Z_n)\right)$. So, by Dominated Convergence Theorem, $\e(e_h(\log_b Z_1))=0$ when $(1/N)\sum_1^Ne_h(\log_b Z_n)$ converges a.s. to $0$. This situation is unlikely to occur (see Section \ref{app}).
\end{remark}

\section{General conditions}\label{gene}

We present in this section the two main results of our paper: Theorem \ref{th1} and Theorem \ref{th2}. Theorem \ref{th1} gives a necessary and sufficient condition ensuring that $(Y_n)$ is a.s. Benford under the assumption that the sequence $(X_n)$ is stationary. Theorem \ref{th2} gives a sufficient condition without constraints on the $X_n$.

\subsection{The first main result}
The proof of Theorem \ref{th1} will use two lemmas. The first one is a simple application of Riesz's summation methods. It is also connected with the N\"ordlund summation methods.

\begin{lemma}\label{lem2}
Let $(a_l)$ be a sequence of complex numbers. Then
\[
\left(\lim_{L\rightarrow +\infty}\frac{1}{L}\sum_{l=1}^{L}a_l=0\right)\Rightarrow\left(\lim_{L\rightarrow +\infty}\frac{1}{L}\sum_{l=1}^{L}\frac{L-l}{L}\ a_l=0\right).
\]
\end{lemma}

\begin{proof}
For every $L\geq 1$
\[
\frac{1}{L}\sum_{l=1}^{L}a_l - \frac{1}{L}\sum_{l=1}^{L}\frac{L-l}{L}\ a_l = \frac{1}{L^2}\sum_{l=1}^{L}la_l.
\]
But the sequences $\left(\frac{1}{L}\sum_{l=1}^{L}a_l\right)$ and $\left(\frac{2}{L^2}\sum_{l=1}^{L}la_l\right)$ are simultaneously convergent and have the same limit when they converge \cite[p. 63]{kuipers}. This concludes the proof.
\end{proof}

The first statement in Lemma \ref{ivdc} is known as the van der Corput Fundamental Inequality \cite[p. 25]{kuipers}. The second statement is a direct consequence of the first one.

\begin{lemma}\label{ivdc}
Let $z_n$ be a complex number for $1\leq n\leq N$. Then, for $1\leq L\leq N$,
\[ \left|\sum_{n=1}^Nz_n\right|^2\leq \frac{N+L-1}{L}\sum_{n=1}^N|z_n|^2+2\operatorname{Re}\left(\frac{N+L-1}{L^2}\sum_{l=1}^{L-1}(L-l)\sum_{j=1}^{N-l}(z_{j+l}\overline{z}_{j})\right). 
\]
In particular, if $x_1,\dots,x_N$ are real numbers,
\[
\left|\frac{1}{N}\sum_{n=1}^Ne^{i(x_1+\dots +x_n)}\right|^2
\leq\frac{2}{L}+2\operatorname{Re}\left(\frac{N+L-1}{L^2N^2}\sum_{l=1}^{L-1}(L-l)\sum_{j=1}^{N-l}\left(e^{i\left(x_{j+1}+\dots+x_{j+l}\right)}\right)\right)
\]
for $1\leq L\leq N$.
\end{lemma}

We are now ready to prove our first main result.

\begin{theorem}\label{th1}
Suppose that $(X_n)$ is stationary. Then the sequence $(Y_n)$ is a.s. Benford in base $b$ if and only if
\[
\forall\, h\in\Z^+\,,\;\lim_{L\rightarrow +\infty}\frac{1}{L}\sum_{l=1}^{L}\e(e_h(\log_b Y_l))=0\,.
\]
\end{theorem}

\begin{proof}
The direct part derives from the Dominated Convergence Theorem. It remains to prove the converse part.

Consider some $h\in\Z^+$. By Lemma \ref{ivdc}, for $1\leq L\leq N$,
\begin{equation}
\left|\frac{1}{N}\sum_{n=1}^Ne_h(\log_b Y_n)\right|^2 \leq\frac{2}{L}+2\operatorname{Re}\left(\frac{N+L-1}{LN}\sum_{l=1}^{L-1}\frac{L-l}{LN}\sum_{j=1}^{N-l}e_h(\log_b Y_{j+1,j+l})\right). \label{eq2}
\end{equation}
Fix $L\geq 1$ and $l\leq L$, set $Z_j= e_h(\log_b Y_{j+1,j+l})$ ($j\geq 0$) and suppose that $(X_n)$ is stationary. Then $(Z_j)$ is stationary and by Birkhoff's Theorem 
\[
\lim_{N\rightarrow +\infty}\frac {1}{N}\sum_{j=1}^{N-l} Z_j=\lim_{N\rightarrow +\infty}\frac {1}{N-l}\sum_{j=1}^{N-l} Z_j=\e^{{\cal B}_l}(Z_0)=\e^{{\cal B}_l}(e_h(\log_b Y_l))\;\;\;\;\mbox{a.s.} 
\]
where ${\cal B}_l$ stands for the $\sigma$-algebra of invariant sets. Hence \eqref{eq2} yields for every $L\geq 1$
\[
\limsup_{N\rightarrow +\infty}\left|\frac{1}{N}\sum_{n=1}^{N}e_h(\log_b Y_n)\right|^2\leq\frac {2}{L}+ \frac{2}{L}\operatorname{Re}\left(\sum_{l=1}^{L} \frac{L-l}{L}\e^{{\cal B}_l}(e_h(\log_b Y_l))\right)\;\;\;\;\mbox{a.s.}.
\]
But $\e\left(\e^{{\cal B}_l}(e_h(\log_b Y_l)\right)=\e(e_h(\log_b Y_l))$. So for every $L\geq 1$
\begin{equation}
\e\left(\limsup_{N\rightarrow +\infty}\left|\frac{1}{N}\sum_{n=1}^{N}e_h(\log_b Y_n)\right|^2\right)\leq\frac {2}{L}+ \left|\frac{2}{L}\sum_{l=1}^{L} \frac{L-l}{L}\e(e_h(\log_b Y_l))\right|.\label{eq3}
\end{equation}

Suppose now that 
\[
\lim_{L\rightarrow +\infty}\frac{1}{L}\sum_{l=1}^{L}\e(e_h(\log_b Y_l))=0.
\] 
Letting $L$ tends to $+\infty$ in \eqref{eq3} and applying Lemma \ref{lem2} with $a_l=\e(e_h(\log_b Y_l))$ give
\[
\e\left(\limsup_{N\rightarrow +\infty}\left|\frac{1}{N}\sum_{n=1}^{N}e_h(\log_b Y_n)\right|^2\right)=0.
\]
This proves that
\[
\limsup_{N\rightarrow +\infty}\left|\frac{1}{N}\sum_{n=1}^{N}e_h(\log_b Y_n)\right|^2=0\;\;\;\;\mbox{a.s.}. 
\] 
So, according to Lemma \ref{weyl}, our proof is completed.
\end{proof}

Surprisingly, the above necessary and sufficient condition appears in a completely different context in \cite{holewijn}. The two following corollaries are direct consequences of Theorem \ref{th1} and Lemma \ref{lem1}. 

\begin{coro}\label{coro1}
Suppose that $(X_n)$ is stationary and that $Y_n$ tends to be Benford in base $b$. Then $(Y_n)$ is a.s. Benford in base $b$.
\end{coro}

\begin{coro}
Suppose that $(X_n)$ is stationary, that the sequence $(Y_n)$ is a.s. Benford in base $b$ and that $(P_{{\cal M}(Y_n)})$ converges weakly to some probability measure $\nu$. Then $\nu$ equals $\mu_b$.
\end{coro}

\begin{remark}
Corollary \ref{coro1} cannot be extended to independent non-stationary r.v. $X_n$. Indeed, consider a Benford r.v. $X_1$ and set $X_n=1$ a.s. for $n\geq 2$. Then $Y_n=X_1$ is a Benford r.v., but $(Y_n)$ is not a.s. Benford. Moreover, there exist stationary sequences $(X_n)$ such that $(Y_n)$ is a.s. Benford and such that  $Y_n$ does not tend to be Benford. For example, suppose that $P_{X}=(1/2)\delta_2+(1/2)\delta_3$ and that $X_n=X$ ($n\geq 1$). Then $(Y_n)$ is a.s. Benford in base $10$ because the sequences $(2^n)$ and $(3^n)$ are natural-Benford in base $10$. But $P_{Y_n}$ does not converge weakly to $\mu_{10}$. We can also consider i.i.d. r.v.'s $X_n$ such that $P_{X_1}=(1/2)\delta_{b^x}+(1/2)\delta_{b^{x+1/2}}$ where $x\in]0;1[$ is irrational. Then $\e(e_h(\log_b Y_n))=e^{2i\pi hnx}$ when $h$ is even and $\e(e_h(\log_b Y_n))=0$ when $h$ is odd. So $(Y_n)$ is a.s. Benford in base $b$ by Theorem \ref{th1} and $Y_n$ does not tend to be Benford in base $b$ since $\left|\,\e(e_h(\log_b Y_n))\right|=1$ when $h$ is even. 
\end{remark}

\subsection{The second main result}

Theorem \ref{th2} below is a direct consequence of Proposition \ref{prop} which is a slight generalization of a (surprisingly little known) result due to Davenport, Erd\H{o}s and Le Veque \cite{davenport}. Proposition \ref{prop} gives a general condition (involving $L^p$-norm) ensuring that the arithmetic mean of bounded random variables converges almost surely to $0$. It will be used in Section \ref{bound}. The result in \cite{davenport} would have been enough to derive the results featuring in Section \ref{bound}, but we think that our more general version may be of interest. We will use the two following lemmas.

\begin{lemma}\label{lem}

Let $(u_N)$ and $(v_N)$ be two sequences of positive numbers and suppose that $v_N>1$ ($N\geq 1$) and that $(v_N)$ is non-decreasing. Consider an increasing sequence $(M_m)$ of positive integers such that
\begin{equation} \label{eq1}
M_{m+1}\geq \frac{v_{M_m}}{v_{M_m}-1}M_m
\end{equation} 
and denote by $U_m$ the arithmetic mean of the numbers $u_{M_m}, u_{M_m+1}, \dots, u_{M_{m+1}-1}$.
Then
\[
\left(\sum_{N=1}^{\infty}\frac{u_Nv_N}{N}<\infty\right)\Rightarrow\left(\sum_{m=1}^{\infty}U_m<\infty\right)\,.
\]
\end{lemma}

\begin{proof}
Fix $m\geq 1$. Then 
\begin{align*}
U_m&=    \frac{1}{M_{m+1}-M_m}\sum_{N=M_m}^{M_{m+1}-1}u_N\\
   &\leq \frac{M_{m+1}}{M_{m+1}-M_m}\sum_{N=M_m}^{M_{m+1}-1}\frac{u_N}{N}\\
   &\leq \sum_{N=M_m}^{M_{m+1}-1}\frac{u_Nv_N}{N}
\end{align*}
because \eqref{eq1} is equivalent to
\[
\frac{M_{m+1}}{M_{m+1}-M_m}\leq v_{M_m}
\]
and the sequence $(v_N)$ is non-decreasing. This concludes the proof since the numbers $\frac{u_Nv_N}{N}$ are positive.
\end{proof}

\begin{lemma}\label{lemm}
Let $(a_n)$ be a sequence of complex numbers satisfying $|a_n|\leq 1$ ($n\geq 1$). Set $b_N=(1/N)\sum_{n=1}^Na_n$ and let $(M_m)$ be an increasing sequence of positive integers such that $\lim_m(M_{m+1}/M_m)=1$. Denote by $c_m$ the arithmetic mean of the numbers $b_{M_m}, b_{M_m+1}, \dots, b_{M_{m+1}-1}$. Then
\[
\lim_{m\rightarrow+\infty}\max_{M_m\leq N<M_{m+1}}\left|b_N-c_m\right|=0\,.
\]
In particular the sequences $(b_N)$ and $(c_m)$ are simultaneously convergent and have the same limit when they converge.
\end{lemma}

\begin{proof}
Fix $N_0$ such that $M_m\leq N_0<M_{m+1}$. Then 
\[
\left|b_{N_0}-c_m\right| \leq\frac{1}{M_{m+1}-M_m}\sum_{N=M_m}^{M_{m+1}-1}\left|b_{N_0}-b_N\right|\,.
\]
But for any $N$ such that $M_m\leq N<M_{m+1}$
\begin{align*}
\left|b_{N_0}-b_N\right| & =\frac{1}{N_0}\left|\sum_{n=1}^{N_0}a_n-\frac{N_0}{N}\sum_{n=1}^{N}a_n\right|\\
                         & \leq  \frac{1}{N_0}\left(\left|\sum_{n=1}^{N_0}a_n-\sum_{n=1}^{N}a_n\right|+\left|\sum_{n=1}^{N}a_n-\frac{N_0}{N}\sum_{n=1}^{N}a_n\right|\right)\\
                         & \leq \frac{2|N_0-N|}{N_0}
\end{align*}
since $|a_n|\leq 1$. Hence
\begin{align*}
\left|b_{N_0}-c_m\right| & \leq \frac{1}{N_0}\frac{2}{M_{m+1}-M_m}\sum_{N=M_m}^{M_{m+1}-1}|N_0-N|\\
                         & \leq 2\frac{M_{m+1}-M_m}{N_0}\,.
\end{align*}
This concludes the proof since $\lim_m(M_{m+1}/M_m)=1$.
\end{proof}

\begin{proposition}\label{prop}

Let $(Z_n)$ be a sequence of complex valued r.v. such that $|Z_n|\leq 1$ ($n\geq 1$) and, for $N\geq 1$, set $T_N=(1/N)\sum_{n=1}^NZ_n$.
Then
\[
\left(\sum_{N=1}^{\infty}\frac{\e|T_N|^p}{N}<+\infty\;\;\;\mbox{ for some }p\geq 1\right)\Rightarrow\left(\lim_{N\rightarrow +\infty} T_N=0\;\mbox{a.s.}\right)\,.
\]
\end{proposition}

\begin{proof}
Fix $p\geq 1$, set $u_N=\e|T_N|^p$ and suppose that $\sum_{N=1}^{\infty}\frac{u_N}{N}<\infty$. Consider a non-decreasing sequence $(v_N)$ of real numbers such that $v_N>1$ ($N\geq 1$), $\lim_Nv_N=\infty$ and $\sum_{N=1}^{\infty}\frac{u_Nv_N}{N}<\infty$ (see \cite{davenport} for an example, among many other possibilities, of one such sequence $(v_N)$). Set $M_1=1$ and for $m\geq 1$ define $M_{m+1}$ as the lowest integer greater than or equals to $\frac{v_{M_m}}{v_{M_m}-1}M_m$. So $(M_m)$ satisfies the hypothesis of Lemmas \ref{lem} and \ref{lemm}. Define the numbers $U_m$ as in Lemma \ref{lem} and set
\[
V_m=\frac{1}{M_{m+1}-M_m}\sum_{N=M_m}^{N=M_{m+1}-1}|T_N|^p\;\;\mbox{ and }W_m=\frac{1}{M_{m+1}-M_m}\sum_{N=M_m}^{N=M_{m+1}-1}|T_N|\,.
\]
According to Lemma \ref{lem} the series  $\sum_mU_m$ converges. But here $U_m=\e(V_m)$ and the r.v.'s $V_m$ are non-negative. So the convergence of $\sum_mU_m$ and the Beppo Levi's Monotone Convergence Theorem imply the a.s. convergence of $\sum_mV_m$ and this proves that $(V_m)$ converge a.s. to $0$. Hence, by Jensen Inequality,  $(W_m^p)$ converges a.s. to $0$. Applying Lemma \ref{lemm} with $c_m=W_m(\omega)$ and $b_N=|T_N(\omega)|$ yields the a.s. convergence to $0$ of $(|T_N|)$. This concludes the proof of Proposition \ref{prop}. 
\end{proof}

Taking $Z_n= e_h(\log_b Y_n)= e_h\left(\sum_1^n\log_b X_k\right)$ 
and combining Section \ref{preli} with Proposition \ref{prop} yields the following theorem.

\begin{theorem}\label{th2}
Set $T_N^{(h)}=(1/N)\sum_{n=1}^Ne_h\left(\sum_1^n\log_b X_k\right)$ and suppose that for all $h\neq 0$ there exists some $p\geq 1$ such that $\sum_{N=1}^{\infty}\frac{\e|T_N^{(h)}|^p}{N}<+\infty$. Then $(Y_n)$ is a.s. Benford in base $b$.
\end{theorem}

\section{Applications of Theorems \ref{th1} and \ref{th2}}\label{appli}

This section is devoted to some applications of Theorems \ref{th1} and \ref{th2}.

\subsection{When $(X_n)$ is stationary}

In this section, we investigate the cases of stationary log-normal r.v.'s, exchangeable r.v.'s and stationary $1$-dependent r.v.'s.

\subsubsection{The case of stationary log-normal r.v.'s}
Suppose that $(\log_b X_n)$ or equivalently $(\log X_n)$ is a gaussian sequence. Then, according to Lemma \ref{lem1}, $Y_n$ tends to be Benford if and only if 
\[
\lim_{n\rightarrow +\infty}\sum_{1\leq k,\,l\leq n}\Cov(\log X_k,\log X_l)=+\infty.
\]
Such a condition holds, for example, when $\log X_n=W_{t_n}$ and where $(W_t)_t$ is a brownian motion or a brownian bridge and $(t_n)$ is any sequence of indexes.

If we suppose in addition that $(\log X_n)$ is stationary, then the above necessary and sufficient condition becomes 
\begin{equation}
\lim_{n\rightarrow +\infty}(n\gamma(1)+2(n-1)\gamma(2)+\dots+2\gamma(n))=+\infty \label{eq4}
\end{equation}
where $\gamma(k)=\Cov(\log X_1,\log X_k)$ $(k=1,\dots,n)$. Therefore Theorem \ref{th1} proves the following result.

\begin{proposition}\label{th3}
If $(\log X_n)$ is a stationary gaussian sequence and satisfies \eqref{eq4}, then $Y_n$ tends to be Benford and $(Y_n)$ is a.s. Benford.

\end{proposition}

The sequence $(X_n)$ statisfies Condition \eqref{eq4} especially when $\gamma(n)\geq 0$ for $n\geq 2$ (for instance when $\log X_n=O_{t_n}$ and $(O_t)_t$ is an Ornstein-Uhlenbeck process) or when the numbers $\gamma(k)$ are summable \cite[p. 215]{doukhan}.

\subsubsection{The case of exchangeable r.v.'s}
In this section, we consider a r.v. $U$ and a sequence of i.i.d.  r.v.'s $(Z_n)$ independent of $U$. For each $n$, we let   
\begin{equation}
X_n=g(U,Z_n) \label{cond1}
\end{equation}
where $g$ is any positive measurable function. 

\begin{proposition}\label{propp}
If the r.v.'s $X_n$ satisfy \eqref{cond1} and if for every positive integer $h$
\begin{equation}
P_U\left(\left\{u : \left|\,\e(e_h(\log_b g(u,Z_1)))\right|<1\right\}\right)=1, \label{eq5}
\end{equation}
then $Y_n$ tends to be Benford in base $b$ and $(Y_n)$ is a.s. Benford in base $b$.
\end{proposition}

\begin{proof}
By Theorem \ref{th1} and Lemma \ref{lem1} we only need to prove that $\lim_n \e(e_h(\log_b Y_n))=0$ for all $h\neq 0$. Fix $h\neq 0$ and $n\geq1$ and suppose that $X_1,\dots,X_n$ satisfy \eqref{cond1}. According to the conditions on $U$ and $(Z_n)$,
\begin{align*}
\e(e_h(\log_b Y_n))&=\e[e_h(\log_b g(U,Z_1))\times\dots\times e_h(\log_b g(U,Z_n))]\\
                   &=\int_{\reel} \e[e_h(\log_b g(U,Z_1))\times\dots\times e_h(\log_b g(U,Z_n))/U=u] dP_U(u)\\
                   &=\int_{\reel} \e[e_h(\log_b g(u,Z_1))\times\dots\times e_h(\log_b g(u,Z_n))] dP_U(u)\\
                   &=\int_{\reel} \e[e_h(\log_b g(u,Z_1))]\times\dots\times \e[e_h(\log_b g(u,Z_n))] dP_U(u)\\
                   &=\int_{\reel} \e[e_h(\log_b g(u,Z_1))]^n dP_U(u).
\end{align*}
Therefore
\[
|\e(e_h(\log_b Y_n))|\leq \int_{\reel} |\e[e_h(\log_b g(u,Z_1))]|^n dP_U(u).
\]
But $|\e[e_h(\log_b g(.,Z_1))]|^n$ converges $P_U$-a.s. to $0$, as $n\rightarrow +\infty$, when \eqref{eq5} holds. This concludes the proof with the  Dominated Convergence Theorem. 
\end{proof}

\begin{remark}
Condition \eqref{eq5} in Proposition \ref{propp} is satisfied in particular when $g(u,Z_1)$ admits a density for $P_U$-almost-all $u$. For example when $Z_1$ admits a density, $g$ is of class $C^1$ and, for $P_U$-almost-all $u$, the set of zeros of $\frac{\partial g}{\partial z}(u,.)$ is finite.
\end{remark}

\subsubsection{The case of stationary $1$-dependent r.v.'s}

For the sake of brevity, we only deal with the case where the r.v.'s $X_n$ are 1-dependent. However, our results can be extended to general $m$-dependence. In this section, we suppose that 
\begin{equation}
X_n=g(Z_n,Z_{n+1}) \label{cond2}
\end{equation}
where the r.v.'s $Z_n$ are i.i.d. and $g$ is any positive measurable function. 

\begin{proposition}\label{proppp}
If the r.v.'s $X_n$ satisfy \eqref{cond2} and if for every positive integer $h$
\begin{equation}
P_{(Z_1,Z_3)}\left(\left\{(z_1,z_3) : \left|\,\e[e_h(\log_b(g(z_1,Z_2)g(Z_2,z_3)]\right|<1\right\}\right)>0, \label{eq6}
\end{equation}
then $Y_n$ tends to be Benford in base $b$ and $(Y_n)$ is a.s. Benford in base $b$.
\end{proposition}

\begin{proof}
Again we only need to prove that $\lim_n \e(e_h(\log_b Y_n))=0$ for all $h\neq 0$. Fix $h\neq 0$ and $n\geq1$ and suppose that $X_1,\dots,X_n$ satisfy \eqref{cond2}. To begin with, we suppose that $n=4m$ where $m$ is some positive integer. Set $V:=(Z_1,Z_3,\dots,Z_{4m+1})$, $v:=(z_1,z_3,\dots,z_{4m+1})$, $\varphi:=e_h(\log_b g)$ and 
\[
\psi(x,z,y):=e_h(\log_b (g(x,z)g(z,y)))=\varphi(x,z)\varphi(z,y)=\log_b(G(x,y,z)).
\] 
Since the $Z_n$ are i.i.d.,
\begin{align*}
\e(e_h(\log_b Y_n))=&\e[\varphi(Z_1,Z_2)\times \varphi(Z_2,Z_3)\times\dots\times \varphi(Z_{4m},Z_{4m+1})]\\
                   =&\int_{\reel^{2m+1}} \e[\varphi(Z_1,Z_2)\times \varphi(Z_2,Z_3)\times\dots\times \varphi(Z_{4m},Z_{4m+1})/V=v] dP_V(v)\\
                   =&\int_{\reel^{2m+1}} \e[\varphi(z_1,Z_2)\times \varphi(Z_2,z_3)\times\dots\times \varphi(Z_{4m},z_{4m+1})] dP_V(v)\\
                   =&\int_{\reel^{2m+1}} \e[\psi(z_1,Z_2,z_3)\times\dots\times \psi(z_{4m-1},Z_{4m},z_{4m+1})] dP_V(v)\\
                   =&\int_{\reel^{2m+1}} \e[\psi(z_1,Z_2,z_3)]\times \dots\times \e[\psi(z_{4m-1},Z_{4m},z_{4m+1})] dP_V(v).
\end{align*}
Besides, the modulus of the expectations appearing in the last integrand are bounded by $1$. Using the fact that the $Z_n$ are i.i.d., we deduce from Fubini's Theorem that 
\begin{align*}
|\,\e(e_h(\log_b Y_n))|\leq & \int_{\reel^{2m+1}} |\,\e[\psi(z_1,Z_2,z_3)]|\times |\,\e[\psi(z_5,Z_6,z_7)]|\times\dots\\
                    &\times |\,\e[\psi(z_{4m-3},Z_{4m-2},z_{4m-1})]| dP_{(Z_1,Z_3,\dots,Z_{4m+1})}(z_1,z_3,\dots,z_{4m+1})\\
                    \leq & \left(\int_{\reel^{2}} |\,\e[\psi(z_1,Z_2,z_3)]|dP_{(Z_1,Z_3)}(z_1,z_3)\right)^{m}.
\end{align*}

Suitably modified, the above calculations still yield
\[
|\,\e(e_h(\log_b Y_n))|\leq \left(\int_{\reel^{2}} |\,\e[\psi(z_1,Z_2,z_3)]|dP_{(Z_1,Z_3)}(z_1,z_3)\right)^{m}
\]
when $n=4m+k$ with $k\in\{1,2,3\}$. To complete our proof we now demonstrate that 
\[
\int_{\reel^{2}} |\,\e[\psi(z_1,Z_2,z_3)]|dP_{(Z_1,Z_3)}(z_1,z_3)<1
\]
when \eqref{eq6} is fulfilled. Set 
\[
A=\left\{(z_1,z_3) : \left|\e[\psi(z_1,Z_2,z_3)]\right|<1\right\}
\] 
and 
\[
A_a=\left\{(z_1,z_3) : \left|\e[\psi(z_1,Z_2,z_3)]\right|\leq a\right\}.
\] 
 
The sequence $(A_{1/n})$ is non-decreasing and $A=\bigcup_nA_{1/n}$. So 
\[
P_{(Z_1,Z_3)}(A)=\lim_{n\rightarrow+\infty} P_{(Z_1,Z_3)}(A_{1/n}).
\]
If $P_{(Z_1,Z_3)}(A)>0$, there exists therefore $a<1$ such that $A_a$ satisfies $P_{(Z_1,Z_3)}(A_a)>0$. Hence
\begin{align*}
\int_{\reel^{2}} |\,\e[\psi(z_1,Z_2,z_3)]|dP_{(Z_1,Z_3)}(z_1,z_3)
=&\int_{\reel^{2}\setminus A_a} |\,\e[\psi(z_1,Z_2,z_3)]|dP_{(Z_1,Z_3)}(z_1,z_3)\\
&+\int_{A_a} |\,\e[\psi(z_1,Z_2,z_3)]|dP_{(Z_1,Z_3)}(z_1,z_3)\\
\leq & P_{(Z_1,Z_3)}(\reel^2\setminus A_a)+aP_{(Z_1,Z_3)}(A_a)
\end{align*}
is strictly less than $1$.
\end{proof}

\begin{remark}
The r.v.'s $X_n$ satisfy Condition \eqref{eq6} in Proposition \ref{proppp} in particular when 
$$
G(x,y,Z_2):=g(x,Z_2)g(Z_2,y) 
$$ 
admits a density for $(x,y)$ in a set of positive $P_{(Z_1,Z_3)}$-measure. For example when $Z_2$ admits a density, $g$ is of class $C^1$ and, for $(x,y)$ in a set of positive $P_{(Z_1,Z_3)}$-measure, the set of zeros of $\frac{\partial G}{\partial z}(x,y,.)$ is finite.

\end{remark}

\subsection{When the $X_n$ are independent}

All the results of this section rely on the following proposition which is a direct consequence of Lemma \ref{lem1} (see also \cite{miller}).

\begin{proposition}\label{prop1} If the $X_n$ are independent, then $Y_n$ tends to be Benford in base $b$ if and only if
$$
\forall\, h\in\Z^+\,,\;\lim_{n\rightarrow +\infty}\prod_{k=1}^n\left|\e(e_h(\log_b X_k))\right|=0\,.
$$
Moreover, if $X_{n_0}$ is a Benford r.v. for some $n_0\geq 1$, then $Y_n$ is a Benford r.v. for all $n\geq n_0$.
\end{proposition}

\subsubsection{A general criterion ensuring that $Y_n$ tends to be Benford}

The following proposition shows that $Y_n$ tends to be Benford in most cases when the $X_n$ are independent. We say that the sequence $(Z_n)$ satisfies 
\begin{itemize}
\item Condition $(C_1)$ if $(Z_n)$ does not admit any subsequence which converges in distribution to a r.v. supported by some set $\{b^{a+\frac{z}{h}} : z \;\mbox{integer}\}$ ($a\in[0,1[$, $h\in\Z^+$);

\item Condition $(C_2)$ if $({\cal M}(Z_n))$ does not admit any subsequence which converges in distribution to a r.v. supported by some set $\{b^{a+\frac{z}{h}} : z=0,1,\dots,h-1\}$ ($a\in[0,1[$, $h\in\Z^+$).
\end{itemize}
 
\begin{proposition}\label{prop4}   
If the $X_n$ are independent and $(X_n)$ possesses a tight subsequence satisfying Condition $(C_1)$, then $Y_n$ tends to be Benford.

The same is true if the $X_n$ are independent and $(X_n)$ possesses a subsequence satisfying Condition $(C_2)$.
\end{proposition}

\begin{proof} 
Suppose that the subsequence $(X_{n_l})_l$ is tight. By Helly Selection Theorem, some subsequence $(X_{n_j})_j$ of $(X_{n_l})_l$ converges in distribution to some r.v. $Z$ and this leads to convergence in distribution of $(\log_b X_{n_j})_j$ to $\log_b Z$. Suppose now that $(X_{n_l})_l$ satisfies Condition $(C_1)$. Then $|\e(e_h(\log_b Z))|<1$ for all $h\in\Z^+$. Fix $h\in\Z^+$. Since
\[
\lim_{j\rightarrow +\infty}|\e(e_h(\log_b X_{n_j}))|=|\e(e_h(\log_b Z))|,
\]
there exists $\varepsilon>0$ and $j_0\geq 1$ such that $|\e(e_h(\log_b X_{n_j}))|\leq 1-\varepsilon$ for $j\geq j_0$.
This yields 
\[
\lim_{m\rightarrow +\infty}\prod_{j=1}^m\left|\e(e_h(\log_b X_{n_j}))\right|=0
\]
which implies
\[
\lim_{n\rightarrow +\infty}\prod_{k=1}^n\left|\e(e_h(\log_b X_k))\right|=0.
\]
Thus, when the $X_n$ are independent, Proposition \ref{prop1} completes the proof of the first assertion.

Consider again a subsequence $(X_{n_l})_l$ of $(X_{n})$. The sequence $({\cal M}(X_{n_l}))_l$ is tight by nature since it is uniformly bounded. Thus some subsequence $({\cal M}(X_{n_j}))_j$ of $({\cal M}(X_{n_l}))_l$ converges in distribution to some r.v. $Z$ with values in $[1,b[$ and this leads to the convergence in distribution of $(\{\log_b X_{n_j}\})_j$ to $\log_b Z$. If $(X_{n_l})_l$ satisfies Condition $(C_2)$, we can conclude with the same arguments as above.
\end{proof}

Notice that the previous proposition is not a consequence of Theorem \ref{th2}. Indeed, the proof of such a result mainly uses the fact that the $X_n$ are independent.

In what follows, we consider r.v.'s $X_n$ with densities. This will allow us to get explicit bounds for the Fourier coefficients of the r.v.'s $\log_b X_n$ and thus to make use of both Proposition \ref{prop1} and Theorem \ref{th2}. 

\subsubsection{A general bound on Fourier coefficients and applications}\label{bound}

Several bounds of the characteristic function are available in the literature, but we are only interested in Fourier coefficients which are easier to investigate. We give below a simple bound which is uniform in $h$. 

\begin{lemma} \label{lem3}
If a r.v. $Z$ admits a density supported by the interval $[0,1[$ and bounded from above by a real $a$, then $|\e(e_h(Z))|\leq \sqrt{1-1/4a^2}$ for all $h\in\Z^+$.
\end{lemma}

\begin{proof}
Let $Z$ be such a r.v and let $h\in\Z^+$ be fixed. If $Z_1$ and $Z_2$ denote two independent r.v's with the same density as $Z$, then
\[
0\leq |\,\e(e_h(Z))|^2=\,\e(e_h(Z_1-Z_2))=\int_{[-1;1]}f(x)\cos(2h\pi x)\,dx
\]
where $f$ is the density of $Z_1-Z_2$. Note that $f\leq a$ too and that $a\geq 1$. For every $l\in[-1,1]$ set
\[
B_l=\{x\in[-1;1] : \cos(2h\pi x)\geq l\}.
\]
Let $L$ be such that the Lebesgue measure of $B_L$ is $1/a$. Then
\[
\int_{B_L^c}f(x)\,dx=\int_{B_L}(a-f(x))\,dx
\]
Hence, if we set $g=a\indic_{B_L}$,
\begin{align*}
\int_{[-1;1]}f(x)\cos(2h\pi x)\,dx
& =\int_{B_L}f(x)\cos(2h\pi x)\,dx+\int_{B_L^c}f(x)\cos(2h\pi x)\,dx\\
& \leq \int_{B_L}f(x)\cos(2h\pi x)\,dx+L\int_{B_L^c}f(x)\,dx\\
& = \int_{B_L}f(x)\cos(2h\pi x)\,dx+L\int_{B_L}(a-f(x))\,dx\\
& \leq \int_{B_L}f(x)\cos(2h\pi x)\,dx+\int_{B_L}(a-f(x))\cos(2h\pi x)\,dx\\
& =\int_{[-1;1]}g(x)\cos(2h\pi x)\,dx.
\end{align*}
Direct calculations and the inequality $\sin x\leq x-x^3/\pi^2\,\,$ ($0\leq x\leq \pi$) give
\[
\int_{[-1;1]}g(x)\cos(2h\pi x)\,dx=\frac{2a}{\pi}\sin \frac{\pi}{2a}\leq 1-\frac{1}{4a^2}\,.
\]
The proof is completed.
\end{proof}

Proposition \ref{propppp} below is an application of Lemma \ref{lem3} in situations where ${\cal M}(X_n)$ admits a density. Its proof uses the following lemma whose proof is elementary.

\begin{lemma}\label{robbins}
Let $x_1,\dots,x_N$ be real numbers. Then
\[ \left|\frac{1}{N}\sum_{n=1}^Ne^{i(x_1+\dots +x_n)}\right|^2= \frac{1}{N}+\frac{2}{N^2}\sum_{1\leq k<n\leq N}\operatorname{Re}\left(e^{i\left(x_{k+1}+\dots+x_{n}\right)}\right). 
\]
\end{lemma}

\begin{proposition}\label{propppp}
Suppose that each r.v. ${\cal M}(X_n)$ admits a bounded density $f_n$. Set $c_n=\sup_{1\leq x<b}f_n(x)$ and $C_N=\max_{1\leq n\leq N}c_n$ ($n\geq 1$, $N\geq 1$). Then $Y_n$ tends to be Benford if $\sum(1/c_n^2)=+\infty$ and $(Y_n)$ is a.s. Benford if $\sum (C_N^2/N^2)<+\infty$. This is in particular the case  when the densities $f_n$ are uniformly bounded.
\end{proposition}

\begin{proof}
Under the above assumptions, the density of $\{\log_b X_n\}$ is bounded by $(b\log b)c_n$. Hence by Lemma \ref{lem3} 
\[
\prod_{k=1}^n\left|\e(e_h(\log_b X_k))\right|\leq \sqrt{\prod_{k=1}^n\left(1-\frac{1}{4(b\log b)^2c_k^2}\right)}
\]
and so 
\[
\log \left(\prod_{k=1}^n\left|\e(e_h(\log_b X_k))\right|\right)\leq \frac{1}{2}\sum_{k=1}^n\log\left(1-\frac{1}{4(b\log b)^2c_k^2}\right).
\]

If $c_k$ does not tend to infinity as $k\rightarrow +\infty$, then a subsequence of $(c_k)$ is bounded and so $\sum \log\left(1-\frac{1}{4(b\log b)^2c_k^2}\right)=-\infty$.
If $\lim_kc_k=+\infty$, then $\log\left(1-\frac{1}{4(b\log b)^2c_k^2}\right)$ is equivalent to $-\frac{1}{4(b\log b)^2c_k^2}$ as $k\rightarrow +\infty$. Thus, in this case, $\sum \log\left(1-\frac{1}{4(b\log b)^2c_k^2}\right)=-\infty$ if and only if $\sum (1/c_k^2)=+\infty$. The proof of the first assertion is completed.

To prove the second assertion, we will make use of Theorem \ref{th2} with $p=2$. Set $T_N=(1/N)\sum_{n=1}^N e_h(\log_b X_n)$ and $A_N=\max_{1\leq n\leq N}|\e(e_h(\log_b X_n))|$ ($N\geq 1$). By Lemma \ref{robbins}, for every $N\geq 1$,

\begin{align*}
\e|T_N|^2 & \leq  \frac{1}{N}+\frac{2}{N^2}\sum_{1\leq k<n\leq N}|\e(e_h(\log_b X_{k+1}))|\dots|\e(e_h(\log_b X_n))|\\
          & \leq \frac{1}{N}+\frac{2}{N^2}\sum_{1\leq k<n\leq N} A_N^{n-k}\\
          & \leq \frac{1}{N}+\frac{2}{N^2}\left((N-1)A_N+(N-2)A_N^2+\dots+A_N^{N-1}\right)\\
          & \leq \frac{1}{N}+\frac{2(N-1)}{N^2(1-A_N)}.
\end{align*}
Hence Theorem \ref{th2} yields that $(Y_n)$ is a.s. Benford when $\sum (1/N^2(1-A_N))<+\infty$. Besides, from Lemma \ref{lem3}, we have $A_N\leq \sqrt{1-1/4(b\log b)^2C_N^2}$. If the non-decreasing sequence $(C_N)$ is bounded, then both $\sum (1/N^2(1-A_N))$ and $\sum (C_N^2/N^2)$ converge. If not, $1/(1-A_N)$ is bounded by $KC_N^2$ where $K$ is a constant. This concludes the proof of Proposition \ref{propppp}. 

\end{proof}

\begin{remark}

There is at least one alternative way to prevent $|\e(e_h(\log_b X_n))|$ from being too close to $1$: limiting the density of ${\cal M}(X_n)$ from below. A mild adaptation of the arguments used in the proof of Lemma \ref{lem3} and of Proposition \ref{propppp} yields that if we replace $c_n=\sup_{1\leq x<b}f_n(x)$ in Proposition \ref{propppp} by $c^{\prime}_n=\min_{1\leq x<b}f_n(x)$, then $Y_n$ tends to be Benford when $\sum c^{\prime}_n=+\infty$. In particular, we also can deduce conditions ensuring that $(Y_n)$ is a.s. Benford.
\end{remark}

\subsubsection{When the r.v. $X_n$ or $\log_b X_n$ is unimodal}

In the next corollary, which is a consequence of Proposition \ref{propppp}, we get a bound for the density of ${\cal M}(X_n)$ by assuming that $\log_b X_n$ or the positive r.v. $X_n$ itself admits a unimodal density (that is to say, a density with a single local maximum). In the event that $\log_b X_n$ admits a unimodal density, we can choose the law of $X_n$ among all the log-stable distributions, and many others since we do not impose the symmetry of the densities. The case where $X_n$ itself admits a unimodal density concerns many usual distributions supported by $]0,+\infty[$: exponential, Fisher-Snedecor, gamma, chi-squared, beta (some of them), Weibull, and so on. Note that Corollary \ref{coro3} does not require any hypothesis on the value or the existence of moments.

\begin{coro}\label{coro3}
Suppose that each r.v. $\log_bX_n$ or each $X_n$ admits a unimodal density $g_n$ and set $d_n=\sup_xg_n(x)$ and $D_N=\max_{1\leq n\leq N}d_n$. Then $Y_n$ tends to be Benford when $\sum(1/d_n^2)=+\infty$ and $(Y_n)$ is a.s. Benford when $\sum (D_N^2/N^2)<+\infty$. This is the case in particular when the densities $g_n$ are uniformly bounded.
\end{coro}

\begin{proof}
Fix $n\geq 1$ and $h\in\Z^+$. Assume that $\log_b X_n$ admits a unimodal density $g_n$ bounded above by $d_n$. Since we only deal with $\left|\e(e_h(\log_b X_n))\right|$ and since we have $\left|\e(e_h(a+\log_b X_n))\right|=\left|\e(e_h(\log_b X_n))\right|$ ($a\in\reel$), we can assume, without loss of generality, that the mode of $\log_b X_n$ is an integer $k_0$. Fix $x\in[0,1[$. The r.v. $\{\log_b X_n\}$ admits a density $g_n^*$ given by
\[
g_n^*(x)=\sum_{-\infty}^{+\infty}g_n(m+x).
\]
Moreover 
\[
g_n(m)\leq g_n(m+x)\leq g_n(m+1) \;\;\;\mbox{ if }m<k_0
\]
and 
\[
g_n(m+1)\leq g_n(m+x)\leq g_n(m) \;\;\;\mbox{ if }m\geq k_0.
\]
Hence
\[
\sum_{-\infty}^{k_0-1}g_n(m)+\sum_{k_0}^{+\infty}g_n(m+1)\leq g_n^*(x)\leq \sum_{-\infty}^{k_0}g_n(m)+\sum_{k_0}^{+\infty}g_n(m)
\]
which yields 
\[
g_n^*(0)-d_n\leq g_n^*(x)\leq g_n^*(0)+d_n.
\]
Integrating over $[0,1[$ the three members of the above formula gives
\[
g_n^*(0)-d_n\leq 1\leq g_n^*(0)+d_n.
\]
Thus $g_n^*(x)\leq 1+2d_n$ which implies that ${\cal M}(X_n)$ admits a density bounded by $(1+2d_n)/\log b$. By Proposition \ref{propppp}, $Y_n$ tends to be Benford when $\sum(1/(1+2d_n)^2)=+\infty$ and $(Y_n)$ is a.s. Benford when $\sum ((1+2D_N^2)/N^2)<+\infty$. This is the case when $\sum(1/d_n^2)=+\infty$ and $\sum (D_N^2/N^2)<+\infty$ respectively, whether $\lim_nd_n=+\infty$ or not.

Assume now that $X_n$ itself admits a unimodal density $g_n$ bounded above by $d_n$. We will assume, without loss of generality, that the mode of $X_n$ is $b^{k_0}$ where $k_0$ is an integer. Fix $x\in[1,b[$. The r.v. ${\cal M}(X_n)$ admits a density $g^*$ satisfying
\[
g_n^*(x)=\sum_{-\infty}^{+\infty}g_n(b^m+x) 
\]
and we conclude in the same spirit as above. 
\end{proof}

Note that the r.v.'s involved in Corollary \ref{coro3} are allowed to converge in distribution to a r.v. supported by some set $\{b^{a+\frac{z}{h}} : z \;\mbox{integer}\}$ ($a\in[0,1[$, $h\in\Z^+$). Hence Corollary \ref{coro3} is not a consequence of Proposition \ref{prop4}.

\section{Appendix} \label{app}

We present here a survey of the main known results on Benford r.v.'s and of results which may be new but are easily deduced from known techniques. All the proofs below use Fourier Analysis and most of them are simpler and shorter than the original ones. See \cite{berger2} for more basic facts on Benford's law.

With some modifications, most of the random variables are close to be Benford  in a sense which will be specified (see e.g. \cite{dumbgen,leemis}). Indeed, if $Z$ is a random variable such that $\lim_{t\rightarrow \infty}\e(\exp(2i\pi tZ))=0$ (this holds in particular when the law of $Z$ is absolutely continuous), then $\lim_{\sigma\rightarrow \infty}\e(e_h(\sigma Z))=0$ for every $h\in\Z^+$. The r.v. $X:=b^{\sigma Z}$ is close to be Benford for sufficiently large $\sigma$ in the sense that $X$ converges in distribution to the Benford's law as $\sigma$ goes to infinity.  This is in particular the case when $X=e^Z$, where $Z$ is an exponential or a Weibull r.v. with a sufficiently small scale parameter. Besides, $Z$ itself is close to be Benford in any base in the particular case where $Z$ is a log-normal or log-Cauchy r.v. provided that the dispersion parameter of the associated normal or Cauchy distribution is sufficiently large (see also Section \ref{known}).

\subsection{Scale-invariance}\label{scale}
The scale-invariance property of the law of the mantissa of a random variable is intrinsic to $\mu_b$.  Historically, it is for this reason that $\mu_b$ has been chosen to depict the First Digit Phenomenon. This property is equivalent to the invariance by translation of the Lebesgue measure on the circle or, what is the same, to the invariance of $\mu_b$ by product modulo $b$.

The following property has been stated, sometimes in a less precise form, by several authors and is proved, as stated below, by Hill \cite{hill2} via techniques involving the $\sigma$-algebra generated by the mantissa function. We give a short and original proof using Fourier analysis.

\begin{proposition}
Let $X$ be a positive random variable. The three following conditions are equivalent :

\begin{enumerate}

\item $X$ is Benford in base $b$;

\item for every $\lambda>0$, $P_{{\cal M}_b(X)}=P_{{\cal M}_b(\lambda X)}$;

\item for some $\lambda>0$ different from any root of $b$, $P_{{\cal M}_b(X)}=P_{{\cal M}_b(\lambda X)}$. 
\end{enumerate}
\end{proposition}

\begin{proof}
Let $X$ be a positive random variable and $\lambda$ be a positive real number. Then, for every $h\in\Z^+$,  
$$
\e(e_h(\log_b(\lambda X)))=e_h(\log_b\lambda)\e( e_h(\log_b X))\,.
$$ 
So, by Lemma \ref{lem1}, Condition 1 implies Condition 2. Moreover, the above formula and Condition 3 imply
$$
\e(e_h(\log_b X))=e_h(\log_b\lambda)\e( e_h(\log_b X))\,.
$$ 
Since $e_h(\log_b\lambda)\neq 1$ when $h\in\Z^+$ and $\lambda$ is not any root of $b$, this implies Condition 1.
\end{proof}

\subsection{Base-invariance and power-invariance}

We must distinguish the notion of base-invariance considered in \cite{knuth} (called {\it base-invariance} in the sequel) from the one studied in \cite{hill2} (called {\it Hill $b$-base-invariance} in the sequel). The first one is defined by
$$
\forall \,b^{\prime}>1,\; \forall \,b^{\prime\prime}>1,  \;P_{{\cal M}_{b^{\prime}}(X)}=P_{{\cal M}_{b^{\prime\prime}}(X)}\,.
$$
The second one is defined by
$$
\forall \,n\in\n^*, \;P_{{\cal M}_{b^{1/n}}(X)}=P_{{\cal M}_b(X)}
$$
where $b>1$ is fixed.

\subsubsection{Base-invariance}\label{base-inv}

Knuth \cite[Exercice 7 pp. 248, 576]{knuth} has proved with skilly calculations that scale-invariance and base-invariance properties are incompatible. Since the scale-invariance property characterizes the Benford random variables, this implies that the Benford random variables cannot satisfy the base-invariance property. The following proposition is a little bit more precise than the Knuth one and its proof is simple.

\begin{proposition}
If $X$ is base-invariant, then $P_X=\delta_1$ and so $X$ cannot be Benford in any base.
\end{proposition}

\begin{proof}
Suppose that $X$ is base-invariant and fix $h\in\Z^+$ and $b^{\prime}>1$. Lemma \ref{lem1} gives 
$$
\forall\,b^{\prime\prime}>1,\;\e(e_h(\log_{b^{\prime}} X))=\e(e_h(\log_{b^{\prime\prime}} X))=\phi(h/\log b^{\prime\prime})
$$
where $\phi$ is the characteristic function of $\log X$. 
Besides, $\phi$ is continuous and satisfies $\phi(0)=1$. Hence, letting $b^{\prime\prime}$ tends to infinity, we get $\e(e_h(\log_{b^{\prime}} X))=1$, which is true for any $h\in\Z^+$. According to the Levy's Theorem on the torus (see Section \ref{ud1}), this implies that $P_{\{\log_{b^{\prime}} X\}}=\delta_0$ and then $P_{{\cal M}_{b^{\prime}}(X)}=\delta_1$. So $P_X$ is supported by the set $\{1, b^{\prime}, (b^{\prime})^2,\dots\}$ and, since $X$ is supposed to be base-invariant, this must be true for every $b^{\prime}>1$. This is impossible unless $X=1$ a.s.. 
\end{proof}

\subsubsection{Hill $b$-base-invariance and power-invariance}

The following proposition has already been proved by Hill \cite{hill2}, by considering the $\sigma$-algebra generated by the mantissa function. However, we give below an original and shorter proof.

\begin{proposition}
A positive absolutely continuous random variable is Hill $b$-base-invariant if and only if it is Benford in base $b$.
\end{proposition}

\begin{proof}
Let $X$ be a positive random variable. Then, for every $h\in\Z^+$ and $n\in\n^*$, 
\[
\e(e_h(\log_{b^{1/n}}X))=\e(e_{hn}(\log_b X))\,.
\]
So, Lemma \ref{lem1} shows that if $X$ is Benford in base $b$, it is also Benford in base $b^{\frac{1}{n}}$ for every $n\in\n^*$. In particular, this implies that $P_{{\cal M}_{b^{1/n}}(X)}=P_{{\cal M}_b(X)}$. Conversely, if we suppose that $X$ is Hill $b$-base-invariant, the above formula gives 
\[
\e(e_h(\log_b X))=\e(e_{hn}(\log_b X))\;\;\;\;(n\geq 1, h\in\Z^+).
\]
Besides, if we assume that $X$ is absolutely continuous, the Riemann-Lebesgue Theorem says that
\[
\lim_n \e(e_{hn}(\log_b X))=0\;\;\;\;(h\in\Z^+).
\] 
Together with Lemma \ref{lem1}, this proves that $X$ is Benford in base $b$.
\end{proof}

Due to Lemma \ref{lem1}, it is easy to verify that $X$ is Benford in base $b$ if and only if the same fact holds for $1/X$. So, since $\log_{b^{1/n}}x=\log_b x^n$ ($x>0$ and $n\in\n^*$), we can rewrite the above proposition as follows.

\begin{proposition}
If a positive random variable $X$ is Benford in base $b$, then, for every $m\in\Z^+$, $X^m$ is also Benford in base $b$. Conversely, if an absolutely continuous positive random variable $X$ satisfies 
$$
\forall \,n\in\n^*, \;P_{{\cal M}_{b}(X^n)}=P_{{\cal M}_b(X)}\,,
$$
then $X$ is Benford in base $b$.
\end{proposition}

\subsection{Product-invariance}

The following proposition generalizes the scale-invariance property because the constant $\lambda$ appearing in Section \ref{scale} can be viewed as a random variable independent of $X$. Besides, it slightly generalizes Theorem 2.3 in \cite{giuliano2}. Note that the authors of \cite{giuliano2} suppose, in their abstract, that $P_X$ is supported by a finite interval, but they do not use this hypothesis in the proof of their theorem which follows the same lines as ours.

\begin{proposition}\label{pprop}
Let $X$ and $Y$ be two independent positive random variables. If $X$ (or $Y$) is Benford in base $b$, then $XY$ is Benford in base $b$ too.
Conversely, if $P_X$ is not supported by any set $\{b^{\frac{z}{h}} : z \;\mbox{integer}\}$ ($h$ positive integer) and if $P_{{\cal M}_b(Y)}=P_{{\cal M}_b(XY)}$, then $Y$ is Benford in base $b$.
\end{proposition}

\begin{proof}
Let $h\in\Z^+$ and suppose that $X$ and $Y$ are independent. Then 
$$
\e(e_h(\log_b(XY)))=\e(e_h(\log_b X))\e(e_h(\log_b Y))\,.
$$
If $X$ is Benford in base $b$, Lemma \ref{lem1} implies $\e(e_h(\log_b X))=0$ and this gives the first part of the proposition. Conversely, if $P_X$ is not supported by any set $\{b^{\frac{z}{h}} : z \;\mbox{integer}\}$ ($h$ positive integer), then $\e(e_h(\log_b X))\neq 1$ and so $\e(e_h(\log_b(XY)))$ and $\e(e_h(\log_b Y))$ cannot be equal unless they are equal to zero.
\end{proof}

\subsection{Mixtures}\label{mixture}

When $X$ and $Y$ are independent, the conditional law of $XY$ given $(Y=a)$ is the law of $aX$. So $P_{XY}$ can be viewed as a mixture of the laws $P_{aX}$ ($a>0$). Theorem 2.3 in \cite{giuliano2} states that, if $X$ is continuous, $P_Y$ and the mixture $P_{XY}$ lead to the same mantissa law in base $b$ if and only if $P_{{\cal M}_b(Y)}=\mu_{b}$ (this is the converse part of Proposition~\ref{pprop} above). Hence such a mixture is rarely the Benford's law (see \cite{hill} for a similar and more sophisticated property). 

But, Proposition~\ref{pprop} also shows that, whatever $P_Y$ is, $P_{{\cal M}(XY)}=\mu_b$ when $X$ is Benford in base $b$ and $X$ and $Y$ are independent. In other words, any mixture (satisfying the above procedure) of laws of Benford random variables in base $b$ is the law of a Benford random variable in base $b$. This property can be extended to general mixtures.

\end{document}